\documentclass[11pt]{amsart}
\usepackage[utf8]{inputenc}
\usepackage[T1]{fontenc}
\usepackage{amsmath}
\usepackage{amsthm}
\usepackage{graphicx}
\usepackage{grffile}
\usepackage{longtable}
\usepackage{wrapfig}
\usepackage{rotating}
\usepackage[normalem]{ulem}
\usepackage{amsmath}
\usepackage{textcomp}
\usepackage{amssymb}
\usepackage{capt-of}
\usepackage{hyperref}
\usepackage{xcolor}
\usepackage{fourier}
\usepackage[shortlabels]{enumitem}

\usepackage{pgfornament}

\newtheorem{thm}{Theorem}

\newtheorem{lemma}[thm]{Lemma}
\newtheorem{prop}[thm]{Proposition}

\theoremstyle{remark}
\newtheorem{defn}[thm]{Definition}
\newtheorem{remark}[thm]{Remark}

\newcommand{\from}{\colon}
\newcommand{\cT}{\mathcal{T}}

\DeclareMathOperator{\diam}{diam}
\DeclareMathOperator{\Lip}{Lip}
\DeclareMathOperator{\Lie}{Lie}
\DeclareMathOperator{\Hom}{Hom}
\DeclareMathOperator{\GL}{GL}
\DeclareMathOperator{\Cone}{Cone}
\DeclareMathOperator{\Nil}{Nil}
\DeclareMathOperator{\opspan}{span}
\newcommand{\N}{\mathbb{N}}
\newcommand{\R}{\mathbb{R}}
\newcommand{\Z}{\mathbb{Z}}

\date{\today}
\title[Carnot rectifiability of sub-Riemannian manifolds]{Carnot rectifiability of sub-Riemannian manifolds with constant tangent}
\author{Enrico Le Donne}

\address[Le Donne]{Department of Mathematics and Statistics, P.O. Box 35,
FI-40014,
University of Jyv\"askyl\"a, Finland}
\email{ledonne@msri.org}

\author{Robert Young}

\address[Young]{Courant Institute of Mathematical Sciences\\
  New York University\\
  251 Mercer St.\\
  New York, NY  10012\\
  USA}
\email{ryoung@cims.nyu.edu}

\thanks{E.L.D. was partially supported by the Academy of Finland (grant
288501
`\emph{Geometry of subRiemannian groups}')
and by the European Research Council
 (ERC Starting Grant 713998 GeoMeG `\emph{Geometry of Metric Groups}').
R.Y.~was supported by NSF grant 1612061.}
\hypersetup{
 pdfauthor={Enrico Le Donne and Robert Young},
 pdftitle={},
 pdfkeywords={},
 pdfsubject={},
 pdfcreator={Emacs 25.1.1 (Org mode 9.0.5)}, 
 pdflang={English}}

\begin{document}
\maketitle
\begin{abstract}
  We show that if $M$ is a sub-Riemannian manifold and $N$ is a Carnot group such that the nilpotentization of $M$ at almost every point is isomorphic to $N$, then there are subsets of $N$ of positive measure that embed into $M$ by bilipschitz maps.  Furthermore, $M$ is countably $N$--rectifiable, i.e., all of $M$ except for a null set can be covered by countably many such maps.
\end{abstract}
\tableofcontents

\newpage
\section{Introduction}

Given two metric measure spaces $M=(M,\mu,d)$ and $N$, we say that $M$ is {\em countably $N$-rectifiable} if there exist countably many biLipschitz embeddings $f_n\from U_n\to M$, with $U_n\subseteq N$ measurable, such that
$$\mu \biggl(M\setminus \bigcup_{n\in \N} f_n (U_n)\biggr) = 0.$$
When $N$ is the Euclidean space $\R^k$, this is the usual notion of a rectifiable set, see \cite{Federer}.

In this paper we will consider the rectifiability of an equiregular sub-Riemannian manifold $M$, with respect to a Carnot group $N$.  These are metric measure spaces in a natural way; since $M$ is equiregular, the Hausdorff dimension of any nonempty open subset of $M$ is the same, say $Q$, and the $Q$--dimensional Hausdorff measure $\mathcal H^Q$ absolutely continuous with respect to any smooth volume form.  We will thus equip all equiregular sub-Riemannian manifolds with Hausdorff measures of appropriate dimension.  We will prove the following theorem relating rectifiability to the nilpotentizations of $M$.
\begin{thm}\label{thm:Teorema_uno}
  Let $M$ be an equiregular sub-Riemannian manifold and let $G$ be a Carnot group.  Then $M$ is countably $G$-rectifiable if and only if for almost every point $p\in M$, the nilpotentization $\Nil(M,p)$ is isomorphic to $G$ as a Lie group.
\end{thm}

For every sub-Riemannian manifold $M$ and every $p\in M$, the nilpotentization $\Nil(M,p)$ of $M$ at $p$ is an invariant related to the tangent cone that encodes the infinitesimal structure of $M$.  Let $\Cone(M,p)$ be the Gromov--Hausdorff tangent of $M$ at $p$, also known as the {\em metric tangent} or \emph{tangent cone}.  Mitchell and Bella\"iche showed that $\Cone(M,p)$ always exists and is a sub-Riemannian manifold and gave a way to calculate $\Cone(M,p)$ by constructing \emph{nilpotent approximations} of a frame of the horizontal distribution of $M$.  The graded Lie algebra generated by these approximations turns out to depend only on $p$ and the horizontal distribution of $M$, and we call it the \emph{nilpotentization} of $M$ at $p$, denoted $\mathfrak{nil}(M,p)$.  (Details of this construction can be found in \cite{bellaiche}, \cite{jeancontrol}, and we will give a sketch in Section~\ref{sec:privileged coords}.)  When $p$ is a regular point of $M$, the dimension of $\mathfrak{nil}(M,p)$ is equal to the topological dimension of $M$ and $\Cone(M,p)$ is isometric to the stratified Lie group $\Nil(M,p)$ with Lie algebra $\mathfrak{nil}(M,p)$.  (See \cite{LeDonne:Carnot} for an introduction to stratified groups and Carnot groups.)

The fact that countable $G$--rectifiability implies that the nilpotentization $\Nil(M,p)$ is almost everywhere isomorphic to $G$ follows from Pansu differentiability.
Indeed, if $U\subset G$ is a subset of positive measure and $f\from U\to M$ is bilipschitz, then $f$ induces bilipschitz maps from $G$  to $\Cone(M,p)$ for a generic $p$.
By work of Pansu, the existence of such a map implies that $\Nil(M,p)$ is isomorphic to $G$ as a Lie group; see Section \ref{sec:proof of main2} for more details.

The main result of this paper is to show that $M$ is countably $G$--rectifiable if $\Nil(M,p)\cong G$ for almost every $p\in M$, where the symbol $\cong$ denotes Lie group isomorphism.
We remark that even if $M$ is equiregular, the nilpotentization $\Nil(M,p)$ may be isomorphic to $G$ almost everywhere but not everywhere, see Example~\ref{sec:Example2}.  To avoid this problem, we prove the following proposition.
\begin{prop}\label{prop:locally_closed_intro}
  Let $M$ be an equiregular sub-Riemannian manifold and let $G$ be a  Lie group.  Let $X=\{p\in M\mid \Nil(M,p)\cong G\}$.  Then $X$ is locally closed.
\end{prop}
In particular, if $\Nil(M,p)\cong G$ almost everywhere, then $\Nil(M,p)\cong G$ on an open subset of $M$ whose complement is a null set.  By restricting to this subset, we may suppose that $\Nil(M,p)\cong G$ for any $p\in M$.

In order to prove that $M$ is countably $G$--rectifiable, we construct a family of biLipschitz maps $f_n\from U_n\to M$, where each $U_n$ is a measurable subset of $G\cong \Cone(M)$.  One difficulty is that in general, there may not be any biLipschitz maps between an open subset of $\Cone(M,p)$ and an open subset of $M$.  The first examples of this were manifolds such that $\Nil(M,p)$ is not constant on any set of positive measure \cite{Varchenko}; we will give a 7--dimensional example in Section \ref{sec:Example1}.  This can also happen when $\Nil(M,p)$ is constant; in \cite{LOW} the authors showed that there are sub-Riemannian nilpotent groups $G$ and $H$ with isomorphic tangent cones that are not locally biLipschitz equivalent.  Therefore, any biLipschitz embedding of a subset $U\subset G$ into $H$ has nowhere dense image.  We thus prove Theorem~\ref{thm:Teorema_uno} by showing that if all tangents of $M$ are isomorphic to $G$, then there is a Cantor set $K\subset G$ of positive measure and a countable collection of biLipschitz embeddings from $K$ to $M$ whose images covers almost all of $M$.  As a corollary, if $G$ and $H$ are sub-Riemannian groups with isomorphic tangent cones, then there are positive-measure subsets $U\subset G$ and $V\subset H$ that are bilipschitz equivalent.

\subsection{Outline of paper}
In Section~\ref{sec:privileged coords}, we recall some results of Bellaïche and Jean on privileged coordinates, nilpotent approximations, and nilpotentizations.   In Section~\ref{sec:constant tangent}, we use these results to approximate distances in a neighborhood of a point $p\in M$ in terms of  distances in $\Cone(M,p)$.  Section~\ref{sec:constant tangent} also contains the proof of Proposition~\ref{prop:locally_closed_intro}.  Section~\ref{sec:bilip maps} contains the proof of Theorem~\ref{thm:Teorema_uno}.  We first use a set of Christ cubes for $G$ to construct the Cantor set $K\subset G$, then 
for each $p\in M$ we
construct biLipschitz embeddings $H_p\from K\to M$ such that   the image $H_p(K)$ has positive density at $p$. We devote Section~\ref{sec:proof of main}   to the proof of Theorem~\ref{thm:Teorema_uno}: In Section~\ref{sec:proof of main1}, we use these maps to prove that $M$ is countably $G$-rectifiable if $G$ is the tangent almost everywhere; while in Section~\ref{sec:proof of main2}, we show that the tangent almost everywhere $G$ if $M$ is countably $G$-rectifiable.

 Finally, in Section~\ref{sec:Example}, we give two examples: an example of an equiregular sub-Riemannian manifold on which the tangent is not constant on any set of positive measure and an example of an equiregular sub-Riemannian manifold on which the tangent is constant almost everywhere but not everywhere.
 
\section{Preliminaries: privileged coordinates, tangent cones, and Bella\"iche's estimates}\label{sec:privileged coords}
Let $M$ be a connected manifold of dimension $n$, let $\Delta\subset TM$ be a sub-bundle of the tangent bundle, and let $g$ be a positive-definite quadratic form defined on $\Delta$.  For each $p$, let $\Delta_1(p)\subset \Delta_2(p)\subset \dots\subset T_pM$ be the subspaces spanned by iterated brackets, so that $\Delta_1(p)=\Delta(p)$ and $\Delta_i(p)$ is spanned by vectors of the form $[X_1,\dots,[X_{k-1},X_k]\dots](p)$, where $1\le k\le i$ and
$X_i$ are vector fields tangent to $\Delta$.  Let $\Delta_i$ be the corresponding ``bundle.''  If $p\in M$ and there is a neighborhood $U$ of $p$ such that $\dim \Delta_i(q)$ is constant on $U$ for each $i$, we say that $p$ is a \emph{regular point}.  If every point is regular, we say that $M$ is equiregular; in this case, the $\Delta_i$'s are sub-bundles of $TM$.

If there is an $i$ such that $TM=\Delta_i$ for some $i$, we say that $\Delta$ is \emph{bracket-generating} and call the triple $M=(M,\Delta,g)$ a \emph{sub-Riemannian manifold} with \emph{rank} $d$ where $d$ is the dimension of the fiber $\Delta(p)$ of the bundle $\Delta$, i.e., $d:=\dim \Delta(p)$ for any $p\in M$.  We define the \emph{step} of $M$ to be the minimal $i$ such that $TM=\Delta_i$.  We can use the quadratic form $g$ to equip $M$ with the path metric  called the \emph{Carnot--Carathéodory metric} or the \emph{sub-Riemannian metric} on $M$.

If $\dim \Delta_i(p)$ is independent of $p$ for all $i$, we say that $M$ is \emph{equiregular} and define $n_i:=\dim \Delta_i(p)$.  By convention, we define $n_0:=0$.  Then the Hausdorff dimension of $M$, with respect to the sub-Riemannian metric, is $Q:=\sum_i i(n_i-n_{i-1})$.  The $Q$--dimensional Hausdorff measure $\mathcal H^Q$ is doubling, Ahlfors regular, and it is absolutely continuous with respect to any smooth volume form.  In particular, a set is null with respect to a Riemannian metric on $M$ if and only if it is null with respect to $\mathcal H^Q$.  

Mitchell and Bellaïche associated a nilpotent Lie algebra, called the tangent Lie algebra, to each point of an equiregular manifold $M$.  Let $p\in M$ and let $\mathbb X=(X_1, \ldots, X_d)$ be a local frame for $\Delta$ defined on a neighborhood $U$ of $p$.  There are a class of coordinate systems for $M$, called {\em privileged coordinate systems}, which induce a grading of the differential operators on $M$; we refer to \cite{bellaiche, jeancontrol} for definitions.  Let $\phi\from U\to \R^n$ be a coordinate system that is privileged at $p$.  The corresponding grading decomposes smooth functions and vector fields on $U$ into functions and fields that are homogeneous with respect to a scaling associated to $\phi$.  The grading respects the bracket operation and if $V$ is a homogeneous vector field, then the coefficients of $\phi_*(V)$ are polynomials.  Each field $X_1,\dots, X_d$ has weight $-1$, and we use the grading to write the $X_i$ as a sum of homogeneous vector fields
$$X_i=X_i^{(-1)}+X_i^{(0)}+\dots$$
such that $X_i(p)=X_i^{(-1)}(p)$.  
We define the \emph{nilpotent approximation} of $X_i$ with respect to $\phi$ as $\hat{X}_i^{\phi,p}:=X_i^{(-1)}$.  We write $\hat{\mathbb{X}}^{\phi, p}=(\hat{X}^{\phi, p}_1,\dots, \hat{X}^{\phi, p}_d)$.

The nilpotent approximations of the $X_i$ generate a Lie algebra $\Lie(\hat{\mathbb{X}}^{\phi, p}).$
This can be equipped with the stratification
$$\Lie(\hat{\mathbb{X}}^{\phi, p})=V_1(\Lie(\hat{\mathbb{X}}^{\phi, p})) \oplus V_2(\Lie(\hat{\mathbb{X}}^{\phi, p})) \oplus \dots,$$
where
$$V_i(\Lie(\hat{\mathbb{X}}^{\phi, p})) = \opspan \{ [\hat{X}_{a_1}^{\phi,p} ,\dots,[\hat{X}_{a_{i-1}}^{\phi,p},\hat{X}_{a_i}^{\phi,p}]\dots] \mid a_j\in \{1,\dots, d\}\}.$$
Bellaïche showed that when $p$ is a regular point, the map $V\mapsto V(p)$ is a linear isomorphism from $\Lie(\hat{\mathbb{X}}^{\phi, p})$ to $T_pM$ \cite[5.21]{bellaiche}.  As in \cite[Section 4]{bellaiche}, we can construct a basis for $\Lie(\hat{\mathbb{X}}^{\phi, p})$ by extending the $X_i$'s to an adapted frame $(Y_1,\dots, Y_n)$ (possibly defined on a smaller neighborhood) such that $X_i=Y_i$ for $i=1,\dots, d$ and each $Y_i$ is a $w_i$--iterated bracket of the $X_i$'s.  Each vector field $Y_i$ has weight $-w_i$ and can be decomposed as a sum of homogeneous vector fields
$$Y_i=Y_i^{(-w_i)}+Y_i^{(-w_i+1)}+\dots.$$
The nilpotent approximations of the $Y_i$, defined as $\hat{Y}_i^{\phi,p}:=Y_i^{(-w_i)}$, form a basis of $\Lie(\hat{\mathbb{X}}^{\phi, p})$.

This Lie algebra depends \emph{a priori} on the choice of $\phi$ and $\mathbb X$, but  if $p$ is a regular point and $\mathbb Z$ is another frame for $\Delta$, then the fields $\hat{Z}^{\phi,p}_i$ are linear combinations of the $\hat{X}^{\phi,p}_i$'s, so
$$\Lie(\hat{\mathbb{X}}^{\phi, p})=\Lie(\hat{\mathbb{Z}}^{\phi, p}).$$
Furthermore, if $\phi'$ is another privileged coordinate system, then by Proposition~5.20 of \cite{bellaiche}, there is a canonical isomorphism
$$\iota_{\phi,\phi'}\from \Lie(\hat{\mathbb{X}}^{\phi, p})\to \Lie(\hat{\mathbb{X}}^{\phi', p})$$
such that
$$\iota_{\phi,\phi'}(\hat{\mathbb{X}}^{\phi, p})=\hat{\mathbb{X}}^{\phi', p}.$$
That is, $\iota_{\phi,\phi'}$ is the unique isomorphism such that $\iota_{\phi,\phi'}(V)(p)=V(p)$ for every vector field $V\in V_1(\Lie(\hat{\mathbb{X}}^{\phi, p}))$, 
We can thus define the \emph{tangent Lie algebra}, also called the \emph{symbol}, of $M$ at $p$ by
$$\mathfrak{nil}(M,p) := \Lie(\hat{\mathbb{X}}^{\phi, p}).$$
This depends on $\phi$, but we suppress $\phi$ in the notation because different choices of $\phi$ lead to canonically isomorphic Lie algebras.  Let $\Nil(M,p)$ be the simply connected stratified Lie group with Lie algebra $\mathfrak{nil}(M,p)$.  We call this the {\em nilpotentization} of $M$ at $p$.

Though $\dim \mathfrak{nil}(M,p) = \dim M = n$, there is no canonical map from $T_pM$ to $\mathfrak{nil}(M,p)$.
Regardless, the arguments above show that if $\tau_{\phi,p}\from \Delta(p) \to \Lie(\hat{\mathbb{X}}^{\phi, p})$ is the linear map such that $\tau_{\phi,p}(X_i(p))=\hat{X}^{\phi,p}_i$, then for any privileged coordinate system $\phi'$, we have $\tau_{\phi',p}=\iota_{\phi,\phi'}\circ \tau_{\phi,p}$.  That is, $\tau_{\phi,p}$ induces an injective linear map $\tau_{p}\from \Delta(p) \to\mathfrak{nil}(M,p)$
whose image is the first stratum $V_1(\mathfrak{nil}(M,p) ) $ of $\mathfrak{nil}(M,p)$.

Mitchell and Bella\"iche proved the following theorem.
\begin{thm}[{\cite[Prop.\ 5.20, Thm.\ 7.36]{bellaiche}}]\label{thm:Bellaiche tangent cones}
  Let $M=(M,\Delta, g)$ be an equiregular sub-Riemannian manifold.
  The tangent cone $\Cone(M,p)$ of $M$ at $p$ exists and is isometric to the sub-Riemannian metric on $\Nil(M,p)$ with horizontal bundle $V_1(\mathfrak{nil}(M,p))$ and quadratic form $(\tau_p)_*(g_p)$.  
\end{thm}

We can also construct $\Cone(M,p)$ directly.  Let $\mathbb{X}$ be an orthonormal frame and let $\hat{\mathbb{X}}^{\phi,p}$ be its nilpotent approximation with respect to a privileged coordinate system $\phi$.  The vector fields $\phi_*(\hat{X}^{\phi,p}_i)$ have polynomial coefficients, so we can extend them to all of $\R^n$.  There is a product structure on $\R^n$ that makes $\R^n$ into a Lie group isomorphic to $\Nil(M,p)$ such that each vector field $\phi_*(\hat{X}^{\phi,p}_i)$ is left-invariant.  If we equip $\R^n$ with the sub-Riemannian structure such that $\phi_*(\hat{X}^{\phi,p}_1),\dots, \phi_*(\hat{X}^{\phi,p}_d)$ are an orthonormal basis for the horizontal distribution, we obtain a left-invariant sub-Riemannian structure which is isometric to $\Cone(M,p)$.  

\bigskip

Given an orthonormal frame $\mathbb X:=(X_1, \ldots, X_{n_1})$ for a sub-Riemannian manifold, we say that an absolutely continuous curve $\gamma\from I\to M$ has {\em controls} $u=(u_1, \ldots, u_{n_1})$ with respect to $\mathbb X$ if 
$$\dot\gamma = u_1 X_1\circ \gamma + \ldots + u_{n_1} X_{n_1}\circ \gamma$$
almost everywhere.
Moreover, we say that the controls of such a 
$\gamma$ are {\em subunit} if 
$$u_1^2+ \ldots+ u_{n_1}^2 \leq1$$ 
almost everywhere.
Let
$$\|u\|_{L_1(L_2)}:=\int_I \sqrt{u_1(t)^2+ \ldots+ u_{n_1}(t)^2}\; dt.$$
Since the $X_i$'s are orthonormal, we have $\ell(\gamma)=\|u\|_{L_1(L_2)}$.

\begin{defn}
  Let $M_1, M_2$ be two manifolds equipped with two frames $\mathbb X$ and $\mathbb Y$, respectively. Assume that the ranks of the frames are the same.  If $\alpha\from I\to M_1$ and $\beta\from I\to M_2$ are two curves, we say that they have \emph{the same controls with respect to $\mathbb X$ and $\mathbb Y$}, respectively (or simply \emph{the same controls} when the frames are clear) if the controls of $\alpha$ with respect to $\mathbb X$ and the controls of $\beta$ with respect to $\mathbb Y$ are equal.   
\end{defn}

We can compare the geometry of the space and its tangent cone by comparing curves with the same controls in $M$ and in $\Nil(M,p)$.
Hereafter, we shall equip $\Nil(M,p)$ with the sub-Riemannian metric defined in Theorem~\ref{thm:Bellaiche tangent cones}, so that it is isometric to $\Cone(M,p)$, and view elements of $\mathfrak{nil}(M,p)$ as left-invariant vector fields on $\Nil(M,p)$.  
We don't claim originality in the following lemma, which follows from the work of
Bella\"iche-Jean.
\begin{lemma}\label{lem:closed curves transfer}
  Let $M$ be an equiregular sub-Riemannian $n$--manifold of step $s$ equipped with an adapted orthonormal frame $\mathbb X=(X_1,\dots, X_d)$.  For $p\in M$, the images $(\tau_p(X_1(p)),\dots, \tau_p(X_d(p)))$ form a left-invariant frame for the horizontal bundle of $\Nil(M,p)$.  We denote this frame by $\tau_p(\mathbb{X}(p))$.
  Let $\bar{p}\in M$.  
  There are \(C_0, L_0>0\), and a compact neighborhood $B_0$ of $\bar{p}$ with the following property.
  Let $p\in B_0$.  Let $\gamma\from [0,1]\to M$ and $\lambda\from [0,1]\to \Nil(M,p)$ be two horizontal curves with the same control $u$ with respect to $\mathbb{X}$ and $\tau_{p}(\mathbb{X}(p))$, respectively.  Suppose that $\|u\|_{L_1(L_2)}\le L_0$ and that $\gamma(0)=p$.

  If $\lambda$ is a closed curve, then 
  $$d_M(\gamma(0),\gamma(1))\le C_0 \|u\|_{L_1(L_2)}^{1+\frac{1}{s}}.$$

  If $\gamma$ is a closed curve, then 
  $$d_{\Nil(M,p)}(\lambda(0),\lambda(1))\le C_0 \|u\|_{L_1(L_2)}^{1+\frac{1}{s}}.$$
\end{lemma}
\begin{proof}
  By  \cite[Theorem 2.3]{jeancontrol}, there are $C_0,L_0>0$ and a compact set $B_0\subset M$ containing a neighborhood of $\bar{p}$ with the following property.  Let $p\in B_0$.  There is a system of privileged coordinates $\phi\from U\to \R^n$ defined on a neighborhood $U$ of $p$ such that $\phi(p)=\mathbf{0}$ and if $\|\cdot\|_p$ is the pseudo-norm
  $$\|(x_1,\dots, x_n)\|_p=\sum_i |x_i|^{\frac{1}{w_i}},$$
  then for any $q$ such that $d(p,q)\le L_0$, we have
  \begin{equation}\label{eq:ball box 1}
    \frac{1}{C_0} \|\phi(q)\|_p \le d_M(p,q)\le C_0 \|\phi(q)\|_p.
  \end{equation}

  We identify $\Nil(M,p)$ with $\R^n$, equipped with the sub-Riemannian structure defined by the vector fields $\phi_*(\hat{X}^{\phi,p}_1),\dots, \phi_*(\hat{X}^{\phi,p}_d)$.  By left-invariance, we may suppose that $\lambda(0)=\mathbf{0}$.  By  \cite[Theorem 2.2]{jeancontrol}, we can choose $C_0$, $L_0$, and $B_0$ so that 
  \begin{equation}\label{eq:ball box 2}
    \frac{1}{C_0} \|v\|_p \le d_{\Nil(M,p)}(0,v)\le C_0 \|v\|_p, \qquad \text{for all }v\in \R^n.
  \end{equation}

 
  By  \cite[(2.14)]{jeancontrol}, (taking $q=p$), there is a $C>0$ such that 
  \begin{equation}\label{eq:Jean 214}
    \|\phi(\gamma(1))-\lambda(1)\|_p\le C \|u\|_{L_1(L_2)}^{1+\frac{1}{s}}.
  \end{equation}
    Suppose that $\lambda$ is a closed curve, so that $\lambda(1)=\mathbf{0}$.  By \eqref{eq:ball box 1} and \eqref{eq:Jean 214},
  $$d_M(\gamma(0),\gamma(1)) \le C_0 \|\phi(\gamma(1))\|_p\le C\cdot C_0 \|u\|_{L_1(L_2)}^{1+\frac{1}{s}}.$$
  Likewise, if $\gamma$ is a closed curve, then $\phi(\gamma(1))=\mathbf{0}$, so
  $$d_{\Nil(M,p)}(\lambda(0), \lambda(1)) \le C_0 \|\lambda(1)\|_p \le C\cdot C_0 \|u\|_{L_1(L_2)}^{1+\frac{1}{s}}.$$
\end{proof}

\section{Manifolds with constant tangent}\label{sec:constant tangent}
In this section, we prove an approximation result for sub-Riemannian manifolds with constant tangent.  


If $M$ is a sub-Riemannian manifold with horizontal distribution $\Delta$ and $\mathbb{X}$ is a frame for $\Delta$, we define $d_{\mathbb{X}}$ to be the sub-Riemannian distance function on $M$ for which $\mathbb{X}$ is an orthonormal frame.
\begin{lemma}\label{lem:transfer lemma}
  Let $(M,\Delta)$ be an equiregular sub-Riemannian manifold and let $G$ be a Carnot group with Lie algebra $\mathfrak{g}$ such that for every $p\in M$, there is an isomorphism $a_p\from \mathfrak{g}\to \mathfrak{nil}(M,p)$.  Suppose that the $a_p$'s vary smoothly in the sense that there is a basis $\mathbb{Y}=(Y_1,\dots, Y_d) \in V_1(\mathfrak{g})$ (i.e., a left-invariant frame of the horizontal bundle) such that if
  $$X_i(p):=\tau_{p}^{-1}(a_p(y_i)),$$
  then $\mathbb{X}=(X_1,\dots, X_d)$ is a smooth frame for $\Delta$, where $\tau_{p}\from \Delta(p) \to\mathfrak{nil}(M,p)$ is the map defined in Section~\ref{sec:privileged coords}.

  For every $p\in M$, there   are \(C, L>0\), and a compact neighborhood $B$    of $p$    with the following property.  Let \(q\in B\) and 
  for $i=1,2$
  let \(\gamma_i\from [0,1]\to M\) be horizontal curves with controls $u_i$ such that \(\gamma_i(0)=q\) and $\|u_i\|_{L_1(L_2)}\le L$.  Let \(\lambda_i\from [0,1]\to G\) be the curves in \(G\) with \(\lambda_i(0)=\mathbf{0}\) and with     controls $u_i$ with respect to $\mathbb{Y}$.  Then
  \begin{equation}\label{eq:transfer controls}
    |d_{\mathbb{X}}(\gamma_1(1),\gamma_2(1))-d_{\mathbb{Y}} (\lambda_1(1),\lambda_2(1))|\le C(\|u_1\|_1+\|u_2\|_1)^{1+\frac{1}{s}}.
  \end{equation}
\end{lemma}
\begin{proof}
  Let $B_0$, $C_0$, and $L_0$ be as in Lemma~\ref{lem:closed curves transfer}.  Let $0<L\le \frac{L_0}{2}$ be small enough that the sub-Riemannian ball $B_{L}(p)$ is contained in $B_0$ and let $B=\overline{B_{L}(p)}$.  This choice ensures that if $q$, $\gamma_1$, and $\gamma_2$ satisfy the hypotheses, then $\gamma_1(1), \gamma_2(1)\in B_0$.   
  
  Let $\alpha\from [0,1] \to G$ be a geodesic from $\lambda_1(1)$ to $\lambda_2(1)$.  We have $\ell(\alpha)\le \|u_1\|_1+\|u_2\|_1 \le 2L$.  The curves $\lambda_1$, $\lambda_2$, and $\alpha$ form a triangle in $G$, and we define $a\from [0,3]\to G$ to be the closed curve that traces the triangle starting at $\lambda_2(1)$, i.e., 
  $$a(t)=\begin{cases} 
    \lambda_2(1-t) & t\in [0,1]\\
    \lambda_1(t-1) & t\in [1,2]\\
    \alpha(t-2) & t\in [2,3].
  \end{cases}$$
  Let $b\from [0,3]\to M$ be the curve with the same controls such that $b(0)=\gamma_2(1)\in B_0$.
  Since $b$ has the same controls as $a$, we have
  $$d_{\mathbb{X}}(b(2),b(3))\le \ell(\alpha)=d_{\mathbb{Y}}(\lambda_1(1),\lambda_2(1)).$$
  Furthermore, $b|_{[0,1]}$ is the reverse of $\gamma_2$ and $b|_{[1,2]}$ is $\gamma_1$, so $b(2)=\gamma_1(1)$.
  Thus, by Lemma~\ref{lem:closed curves transfer}, 
  \begin{align*}
    d_{\mathbb{X}}(\gamma_1(1), \gamma_2(1)) 
    &\le d_{\mathbb{X}}(b(2),b(3)) + d_{\mathbb{X}}(b(3), \gamma_2(1)) \\ 
    &\le d_{\mathbb{Y}}(\lambda_1(1),\lambda_2(1)) + C_0 \ell(b)^{1+\frac{1}{s}}\\ 
    &\le d_{\mathbb{Y}}(\lambda_1(1),\lambda_2(1)) + C_0 \left(2(\|u_1\|_1+\|u_2\|_1)\right)^{1+\frac{1}{s}}.
  \end{align*}

  This proves one inequality.  To prove the other inequality, we apply the same procedure with $\gamma$ and $\lambda$ switched.  That is, we connect $\gamma_1(1)$ and $\gamma_2(1)$ by a geodesic to construct a curve $a\from [0,3]\to M$ such that $c(0)=c(3)=\gamma_2(1)$, $c(1)=q$, and $c(2)=\gamma_1(1)$.  Let $b\from [0,3]\to M$ be the curve with the same controls such that $b(0)=\lambda_2(1)$, so that $b(1)=\mathbf{0}$ and $b(2)=\lambda_1(1)$.    As above,
  $$d_{G}(b(2),b(3))\le d_{\mathbb{X}}(\gamma_1(1),\gamma_2(1)),$$
  and by Lemma~\ref{lem:closed curves transfer},
  \begin{multline*}
    d_{\mathbb{Y}}(\lambda_1(1), \lambda_2(1)) \le d_{\mathbb{X}}(\gamma_1(1),\gamma_2(1)) + C_0 \ell(b)^{1+\frac{1}{s}}\\
    \le d_{\mathbb{X}}(\gamma_1(1),\gamma_2(1)) + C_0 \left(2(\|u_1\|_1+\|u_2\|_1)\right)^{1+\frac{1}{s}}.
  \end{multline*}
\end{proof}

Next, in Lemma~\ref{lem:isomorphic frames}, we shall prove   that a sub-Riemannian manifold $M$ has local frames satisfying Lemma~\ref{lem:transfer lemma} if and only if its tangent Lie algebra is constant.  We first need some notation and the following Lemma~\ref{lem:space of forms}. 

Let $V$ be a finite-dimensional vector space and let $F_V:=\Hom(V\wedge V,V)$, 
seen as the set of alternating bilinear maps.  This is a finite-dimensional vector space and thus an algebraic variety.
Thus, if $\Psi\in F_V$ satisfies the Jacobi identity, then $(V,\Psi)$ is a Lie algebra.  For any Lie algebra $\mathfrak{g}$ of the same dimension, we define
$$E_{\mathfrak{g}} := \{\Psi\in F_V\mid (V,\Psi)\cong \mathfrak{g}\},$$
where the symbol $\cong$ denotes Lie algebra isomorphism.
\begin{lemma}\label{lem:space of forms}
  Let $\mathfrak{g}$ and $E_{\mathfrak{g}}$ be as above.  Then $E_{\mathfrak{g}}$ is a smooth submanifold of $F_V$.
\end{lemma}
\begin{proof}  
  Let $V, W$ be vector spaces, let $A\from V\to W$ be a linear isomorphism, and let $\Psi\in \Hom(V \wedge V, V)$.   We define the push-forward of $\Psi$ by $A$ to be the map $A_*\Psi\in \Hom(W \wedge W, W)$,
  $$(A_*\Psi)(w_1,w_2) = A\Psi(A^{-1}w_1,A^{-1}w_2).$$
  If $B\from W\to X$ is a linear isomorphism, then $A_*B_*\Psi=(AB)_*\Psi$.  That is, this defines a left action of $\GL(V)$ on $F_V$.  
  If $(V,\Psi)$ is a Lie algebra with bracket $\Psi$, then $(W, A_*\Psi)$ is a Lie algebra and $A$ is an isomorphism from $(V,\Psi)$ to $(W,A_*\Psi)$.
  
  Fix some  $\Psi_0\in E_{\mathfrak{g}}$ and let $f\from \GL(V) \to F_V$ be the orbit map $f(A):=A_* \Psi_0$.
  If $\Psi \in E_{\mathfrak{g}}$, then there is a Lie algebra isomorphism $A\from (V,\Psi_0) \to (V,\Psi)$, so $A_*\Psi_0=\Psi$.  It follows that $E_{\mathfrak{g}}=\GL(V)_* \cdot \Psi_0
  = f ( \GL(V)) $.   
      
  The map $f$ has polynomial coefficients, so it is a regular map (in the sense of algebraic geometry) on an affine variety.  Consequently, its image is locally closed in the Zariski topology (see for instance \cite[3.16]{Harris_Algebraic_geometry}).  In particular, $E_\mathfrak{g}$ is a smooth submanifold except possibly on a singular set of lower dimension. In particular the singular set is a proper subset of $E_\mathfrak{g}$.
        Since $\GL(V)$ acts transitively on $E_\mathfrak{g}$, the space $E_\mathfrak{g}$ is homogeneous, so  the singular set is empty. 
        Thus $E_\mathfrak{g}$ is a smooth submanifold.
\end{proof}
If $\mathfrak{g}$ is a stratified Lie algebra and $V=V_1(V)\oplus V_2(V)\oplus \dots$ is a graded vector space with $\dim V_i(V)=\dim V_i(\mathfrak{g})$, it likewise holds that
\begin{equation}\label{eq:E strat}
  E^{\mathrm{s}}_{\mathfrak{g}} = \{\Psi\in F_V\mid (V,\Psi)\cong_{\mathrm{s}} \mathfrak{g}\}
\end{equation}
is a smooth submanifold of $F_V$.  Here the symbol $\cong_{\mathrm{s}}$ denotes a stratified Lie algebra isomorphism, i.e., an isomorphism $A\from (V,\Psi)\to \mathfrak{g}$ such that $A(V_i(V))=V_i(\mathfrak{g})$.

Lemma~\ref{lem:space of forms} allows us to prove the following. 
\begin{lemma}\label{lem:isomorphic frames}
  Let $(M,\Delta)$ be an equiregular sub-Riemannian manifold.  Suppose that there is a     Lie algebra $\mathfrak{g}$ such that $\mathfrak{g}\cong\mathfrak{nil}(M,p)$ for every point $p\in M$.  
  For any $p\in M$, there are a neighborhood $U$ of $p$, a family of isomorphisms $a_q\from \mathfrak{g}\to \mathfrak{nil}(M,q)$, and a basis $Y_1,\dots, Y_d\in V_1(\mathfrak{g})$ satisfying the assumptions of Lemma~\ref{lem:transfer lemma} for $M=U$.
\end{lemma}
\begin{proof}
  We first construct a grading of $\R^n$ and a family of smoothly varying forms $\Psi_q\in F_{\R^n}$ such that $(\R^n,\Psi_q)\cong_{\mathrm{s}} \mathfrak{nil}(M,q)$; i.e., $\Psi_q\in E^{\mathrm{s}}_{\mathfrak{g}}$.  By Lemma~\ref{lem:space of forms}, $E^{\mathrm{s}}_{\mathfrak{g}}$ is a smooth manifold, so we may apply the Implicit Function Theorem to produce the desired isomorphisms.  
  
  
  Let $U$ be a neighborhood of $p$ such that there is an adapted frame $(W_1,\dots, W_n)$ defined on $U$.  As in \cite[2.2.2]{jeancontrol}, for every $q\in U$, we can use the $W_i$'s to produce a system $\phi_q$ of exponential coordinates such that $\phi_q$ is privileged at $q$ and varies smoothly with $q$.  Let $\hat{W}^q_i:=\hat{W}^{\phi_q,q}_i$ be the nilpotent approximation of the $W_i$'s at $q$.  By the results of Section~\ref{sec:privileged coords} the span of the $\hat{W}^q_i$ is a Lie algebra of vector fields which is canonically isomorphic to $\mathfrak{nil}(M,q)$; let 
  $$\mathfrak{g}_q:= \langle \hat{W}^{q}_1,\dots, \hat{W}^{q}_n \rangle$$
  and let $\iota_q\from \mathfrak{g}_q\to \mathfrak{nil}(M,q)$ be the canonical isomorphism.  This is the unique isomorphism such that $\iota_q(V)=\tau_q(V(q))$ for all $V\in V_1(\mathfrak{g}_q)$.  

  For all $q\in U$, we define the basis $\widehat{\mathbb{W}}^q:=(\hat{W}^{q}_1,\dots, \hat{W}^{q}_n)$, and for all $v=(v_1,\dots, v_n) \in \R^n$, let
  $$\widehat{\mathbb{W}}^qv:=v_1 \hat{W}^{q}_1 + \dots + v_n \hat{W}^{q}_n\in  \mathfrak{g}_q.$$  
  This induces an linear isomorphism from $\R^n$ to $\mathfrak{g}_q$.  Let $\Psi_q:=((\widehat{\mathbb{W}}^q)^{-1})_*[\cdot,\cdot]_q\in F_{\R^n}$, i.e.,
  $$\Psi_q(v,w)=(\widehat{\mathbb{W}}^q)^{-1}\left[\widehat{\mathbb{W}}^qv, \widehat{\mathbb{W}}^qw\right]_q,$$
  so that $(\R^n,\Psi_q)$ is a stratified Lie algebra isomorphic to $\mathfrak{g}_q$, with strata
  $$V_i(\R^n):=(\widehat{\mathbb{W}}^q)^{-1} V_i(\mathfrak{g}_q).$$
  That is, $V_i(\R^n)$ is the subspace spanned by the $n_{i}$th through $(n_{i+1}-1)$th coordinate vectors.  The coordinates of $\Psi_q$ in $F_{\R^n}$ are the structure coefficients of $\mathfrak{g}_q$ with respect to $\widehat{\mathbb{W}}^q$, so $\Psi_q$ varies smoothly with $q$. 

  Let $S\subset \GL_n(\R)$ be the subgroup of block-diagonal matrices that preserve the grading of $\R^n$ and let $E^{\mathrm{s}}_{\mathfrak{g}}$ be as in \eqref{eq:E strat}.  For every $q$, we have $(\R^n,\Psi_q) \cong_{\mathrm{s}} \mathfrak{g},$ so $\Psi_q\in E^{\mathrm{s}}_{\mathfrak{g}}$; indeed, $E^{\mathrm{s}}_{\mathfrak{g}}=S_*\Psi_q$.  By the remark after Lemma~\ref{lem:space of forms},  $E^{\mathrm{s}}_{\mathfrak{g}}$ is a smooth submanifold of $F_{\R^n}$.  Let $k=\dim E^{\mathrm{s}}_{\mathfrak{g}}$.  

  Let $f\from S\to E^{\mathrm{s}}_{\mathfrak{g}}$ be the map $f(A):=A_*\Psi_p$.  This map $f$ is smooth and surjective, and its derivative $Df$ has constant rank, so by Sard's theorem, $Df$ has rank $k$ everywhere.  By the Implicit Function Theorem, there is a neighborhood $T\subset E^{\mathrm{s}}_{\mathfrak{g}}$ of $\Psi_p$ and a smooth section $\alpha\from T\to S$ such that $\alpha(\Psi_p)=I$ and $f(\alpha(\Psi)) = \alpha(\Psi)_* \Psi_p = \Psi$ for all $\Psi \in T$.  That is, $\alpha(\Psi)$ is a stratified isomorphism from $(\R^n, \Psi_p)$ to $(\R^n, \alpha(\Psi)_* \Psi_p) = (\R^n, \Psi)$.  We fix a stratified isomorphism $\beta \from \mathfrak{g} \to (\R^n, \Psi_p)$, and define
  $$a_q :=\iota_q\circ \widehat{\mathbb{W}}^q  \circ \alpha(\Psi_q) \circ \beta$$
  for all $q\in U\cap \psi^{-1}(T)$.  This is a family of stratified isomorphisms from $\mathfrak{g}$ to $\mathfrak{g}_q$, and for any $Y\in V_1(\mathfrak{g})$, 
  $$\tau_q^{-1}\circ a_q(Y)= \tau_q^{-1}\left(\iota_q\left(\widehat{\mathbb{W}}^q\alpha(\Psi_q) \beta(Y)\right)\right) =\left(\widehat{\mathbb{W}}^q \alpha(\Psi_q) \beta(Y)\right)(q).$$
  This depends smoothly on $q$, so any basis $Y_1,\dots, Y_d\in V_1(\mathfrak{g})$ satisfies the assumptions of Lemma~\ref{lem:transfer lemma}, as desired.
\end{proof}

The following lemma is a further consequence of Lemma~\ref{lem:space of forms}.
We say that a subset $X$   is {\em locally closed} if, for all $p\in X$, there is a neighborhood $V$ of $p$ such that $X\cap V$ is relatively closed in $V$.  
\begin{prop}\label{prop:locally_closed}
  Let $M$ be an equiregular sub-Riemannian manifold and let $\mathfrak{g}$ be a  Lie algebra.  Let $X=\{p\in M\mid \mathfrak{nil}(M,p)\cong \mathfrak{g}\}$.  Then $X$ is locally closed.
\end{prop} 
\begin{proof} 
  Let $p\in X$.  As in Lemma~\ref{lem:isomorphic frames}, let $U$ be a neighborhood of $p$ equipped with an adapted frame $(W_1,\dots, W_n)$ and a family of privileged coordinate systems $\phi_q$ at $q$ that varies smoothly with $q$.  This induces a map $\psi\from U\to F_{\R^n}$ such that $X\cap U=\psi^{-1}(E_{\mathfrak{g}})$, and the preimage of a locally closed set is locally closed.  
\end{proof}
In particular, if $X$ is dense in $M$, then for all $p\in X$, there is a neighborhood $V$ of $p$ such that $X\cap V$ is relatively closed in $V$ and thus $V\subset X$; i.e., $X$ is an open subset of $M$.  Note that $X$ need not be all of $M$.  See Section~\ref{sec:Example2} for an example where $M$ is equiregular and $X$ is almost all of $M$ but $X\ne M$.  

\section{Bilipschitz maps with images of positive measure}\label{sec:bilip maps}
To prove Theorem~\ref{thm:Teorema_uno}, we will first construct a family of bilipschitz maps from a subset of $G$ to $M$ whose images have positive measure.  In Section~\ref{sec:proof of main}, we will cover almost all of $M$ by countably many such images.

We will show the following proposition.
\begin{prop}\label{prop:bilipschitz images}
  Let $G$ be a Carnot group with Lie algebra $\mathfrak{g}$.  
  Let $(M,\Delta)$ be a sub-Riemannian manifold such that $\mathfrak{g}\cong\mathfrak{nil}(M,p)$ for every point $p\in M$.  For every $p\in M$ there is a subset $K\subset G$ and a bilipschitz embedding \(H\from K\to M\) such that $H(K)$ has positive density at $p$.
\end{prop}

Let $\mathbb{Y}$ be a left-invariant frame for $V_1(\mathfrak{g})$.  By Lemma~\ref{lem:isomorphic frames}, there is a local frame $\mathbb{X}$ for $\Delta$ defined on a neighborhood $U$ of $p$ that satisfies the hypothesis of Lemma~\ref{lem:transfer lemma}.  We equip $G$ with the metric $d_{\mathbb{Y}}$ and equip $U$ with the metric $d_{\mathbb{X}}$.  There is no loss of generality here because any two Carnot metrics on $G$ are bilipschitz equivalent and, after possibly passing to a smaller neighborhood, $d_{\mathbb{X}}$ is bilipschitz equivalent to the sub-Riemannian metric on $U$.

We divide the proof of the novel direction of Theorem~\ref{thm:Teorema_uno} into three parts.  We will first construct a Cantor set $K$ by using a set of Christ cubes for $G$, then construct the map $H$ that embeds $K$ into $M$.  This proves Proposition~\ref{prop:bilipschitz images}.  Finally, we will use the maps produced by Proposition~\ref{prop:bilipschitz images} to show that $M$ is  countably $G$-rectifiable.

\subsection{Constructing $K$}
\label{sec:org2c7abaa}

If $S_i$ is a collection of sets, we denote the disjoint union of the $S_i$'s by
$$\sqcup_i S_i=\{(s,i)\mid i\in \Z, s\in S_i\}.$$
We will often refer to elements of a disjoint union and elements of the constituent sets interchangeably.  

Let $X$ be a metric space and let $\mu$ be a measure on $X$ that is Ahlfors regular.  A \emph{cubical patchwork} or \emph{set of Christ cubes} for $X$ is a collection of nested partitions of $X$, analogous to the tilings of $\R^n$ by dyadic cubes.  That is, for each $i$, there is a partition $\Delta_i$ of $X$ (a set of subsets of $X$ that are pairwise disjoint and whose union is all of $X$) that satisfies the following properties.  There are $\sigma\in (0,1)$,  $a_0>0$, $\eta>0$, and $0<C_1<C_2<\infty$ such that:
\begin{enumerate}[1.]
\item 
  If $Q\in \Delta_k$ and $Q'\in \Delta_l$, with $k\le l$, then either $Q'\subset Q$ or $Q\cap Q'=\emptyset$.  
\item Every $Q\in \Delta_k$ has a \emph{parent} $P(Q)\in \Delta_{k-1}$ such that $Q\subset P(Q)$.  
\item Every $Q\in \Delta_k$ contains a ball of radius $C_1 \sigma^{k}$.
\item $\diam Q\le C_2\sigma^{k}$, for every $Q\in \Delta_k$.
\item  
  For all $t>0$ and all $Q\in \Delta_k$, let
  $$\partial_tQ:=\{x\in Q\mid d(x,X\setminus Q)<t\sigma^k\}\cup  \{x\in X\setminus Q\mid d(x,Q)<t\sigma^k\}.$$
  Then for any $0<t\le 1$, 
  \begin{equation}\label{eq:boundary size}
    \mu(\partial_tQ) \le a_0 t^\eta \mu(Q).
  \end{equation}
\end{enumerate}
We call the elements of $\Delta_i$ \emph{cubes}, and we let $\Delta:=\sqcup_i \Delta_i$ be the disjoint union of the $\Delta_i$'s.  Every Ahlfors regular metric space admits a cubical patchwork \cite{MR1009120, ChristTb} for each $\sigma\in (0,1)$, and if $X$ is Ahlfors $d$--regular, then $\mu(Q)\approx (\sigma^k)^d$ for all $Q\in \Delta_k$.  

Let $\Delta=\sqcup_{i\in\Z} \Delta_i$ be a cubical patchwork for $G$ with $\sigma=\frac{1}{2}$.  For any $Q\in \Delta$, let
$$\Delta_i(Q):=\{R\in \Delta_i\mid R\subset Q\}$$
and let $\Delta(Q):=\sqcup_{i\in\Z} \Delta_i(Q)$.  

We use $\Delta$ to construct a Cantor set in $G$.  Let $\tau>0$ be a small number to be determined later and let $Q_0\in \Delta_0$.  Let \(K_0=Q_0\), and for \(k\ge 0\), let
\begin{equation}\label{eq:construct K}
  K_{k+1}=K_k \setminus \bigcup_{Q\in \Delta_k(Q_0)} \partial_{\tau 2^{-\frac{k}{2s}}}Q.
\end{equation}
Note that $\partial_{t}Q$ is open for any $t$ and $Q$ and that $K_1=Q_0\setminus \partial_{\tau}Q_0$ is closed, so $K_k$ is closed for all $k\ge 1$.  Let \(K:=\bigcap_i K_i\).  This is a compact, totally disconnected set.

\begin{lemma}
  When $\tau>0$ is sufficiently small, then $\mu(K)>0$.  
\end{lemma}
\begin{proof}
  For any $k\ge 0$,
  \begin{align*}
    \mu(K_k)-\mu(K_{k+1})
    &\le \mu\left(\bigcup_{Q\in \Delta_k(Q_0)} \partial_{\tau 2^{-\frac{k}{2s}}}Q\right)\\
    &\stackrel{\eqref{eq:boundary size}}{\le} \sum_{Q\in \Delta_k(Q_0)} a_0\cdot (\tau 2^{-\frac{k}{2s}})^\eta \mu(Q).
  \end{align*}
  Since $\Delta_k(Q_0)$ is a partition of $Q_0$, we have
  $$\mu(K_k)-\mu(K_{k+1})\le a_0 \tau^\eta 2^{-\frac{k \eta}{2s}} \mu(Q_0)$$
  and
  $$\mu(K_0)-\mu(K_k)\le \sum_{i=0}^{k-1} a_0 \tau^\eta 2^{-\frac{k \eta}{2s}} \mu(Q_0) \le a_0 \tau^\eta (1-2^{-\frac{\eta}{2s}})^{-1}\mu(Q_0).$$
  As $\tau$ goes to zero, this difference goes to zero, so $\lim_{\tau\to 0} \mu(K) = \mu(K_0)=\mu(Q_0)>0$.
\end{proof}

\subsection{Constructing $H$}
We construct $H$ by constructing a metric tree $\cT$ and a map $A\from \cT\to G$ that sends the ends of $\cT$ to the points of $K$ bijectively.  Each point $q\in K$ then corresponds to a curve $\alpha_q$ in $\cT$, and there is a map $F\from \cT\to M$ so that the controls of $F\circ \alpha_q$ with respect to $\mathbb{X}$ are a rescaling of the controls of $A\circ \alpha_q$ with respect to $\mathbb{Y}$.  We then use Lemma~\ref{lem:transfer lemma} to show that $H=F\circ A^{-1}$ satisfies the desired properties.  

Let 
$$\Lambda:=\{Q\in \Delta(Q_0)\mid Q\cap K\ne\emptyset\}$$
and let $\cT$ be the rooted tree with one vertex $v_Q$ for every cube $Q\in\Lambda$.  We let $v_0=v_{Q_0}$ be the root of $\cT$, and for each $Q\in \Delta(Q_0)$ with $Q\ne Q_0$, we connect $Q$ to its parent $P(Q)$.  Let $V(\cT)$ be the vertex set of $\cT$ and let $V_i(\cT)\subset V(\cT)$ be the $i$th generation of the tree; i.e., the set of vertices corresponding to elements of $\Delta_i$.  For every vertex $v\in V(\cT)$, we denote the corresponding cube by $[v]\in \Lambda$.

We equip $\cT$ with a path metric so that the edges between $V_i(\cT)$ and $V_{i-1}(\cT)$ have length $2^{-i}$.  Then \(d(v_0,v)=1-2^i\) for every \(v\in V_i(\cT)\).  Let \(\overline{\cT}\) be the metric completion of \(\cT\).  This completion consists of the union of \(\cT\) with a Cantor set, which we denote by \(J\).  

Let $\rho_i\from J\to V_i(\cT)$ be the map that sends $x\in J$ to the $i$th generation vertex that is closest to $x$.  For every $x\in J$, the path starting at $\rho_0(x)=v_0$ that passes through $\rho_1(x), \rho_2(x)$, and so on is a geodesic of length $1$ connecting $v_0$ to $x$.  For each $i$, we have $[\rho_{i+1}(x)]\subset [\rho_i(x)]$.  

Next, we define a Lipschitz map \(A\from \cT\to G\).  For each $Q$, we send $v_Q$ to a point $A(v_Q)\in Q$ and send the edge from $Q$ to $P(Q)$ to a geodesic in $G$.  For every $i>0$ and $Q\in V_i(\cT)$, we have $A(v_Q)\in Q\subset P(Q)$, so
$$d(A(v_Q),A(v_{P(Q)}))\le \diam P(Q) \le  C_2 2^{-i} = 2C_2 d(v_Q,v_{P(Q)}).$$
Thus $A$ is Lipschitz.  It extends to a map \(\overline{A}\from \overline{\cT}\to G\), and for every $x\in J$,
\begin{equation}\label{eq:bartau descending chain}
  \bigcap_{i\in \N} [\rho_i(x)]=\{\overline{A}(x)\}.
\end{equation}
The intersections $K\cap [\rho_i(x)]$ are all nonempty, so since $K$ is closed, $\overline{A}(x)\in K$.  

For every $x,y \in J$, let 
\begin{equation}\label{eq:define i}
  i(x,y):=\sup\{i\in \Z \mid \rho_i(x)=\rho_i(y)\}
\end{equation}
and let $a(x,y)=\rho_{i(x,y)}(x)\in V(\cT)$, so that $[a(x,y)]$ is the minimal cube containing both $\overline{A}(x)$ and $\overline{A}(y)$.  Then 
\begin{equation}\label{eq:i is log}
  d(x,y)=d(x,a(x,y))+d(a(x,y),y)=2^{-i(x,y)+1}.
\end{equation}
This is an ultrametric on $J$.  

\begin{lemma}\label{lem:barA biholder}
  The restriction $\overline{A}|_J$ is a bijection from $J$ to $K$ so that for all $x,y\in J$, 
  $$\frac{1}{8} \tau d(x,y)^{1+\frac{1}{2s}}\le d(\overline{A}(x),\overline{A}(y)) \le 2C_2 d(x,y).$$
\end{lemma}
\begin{proof}
  First, we show that $\overline{A}(J)=K$.  If $k\in K$, then for each $i\ge 0$, there is a cube $R_i\in \Lambda_i$ such that $k\in R_i$.  Since $R_i$ and $R_{i+1}$ intersect, $R_{i+1}\subset R_i$, so the vertices $v_{R_0},v_{R_1},\dots$ form a path starting at the root of $\cT$; this path converges to a point $x\in J$, and 
  $$\overline{A}(x)=\bigcap_i R_i=\{k\}.$$
  
  We previously showed that $A$ is $2C_2$--Lipschitz, so $d(\overline{A}(x),\overline{A}(y))\le 2C_2 d(x,y)$ for all $x,y\in J$.  Suppose that $x_1,x_2\in J$ and $x_1\ne x_2$, so that $i(x_1,x_2)<\infty$.  Let $i=i(x_1,x_2)$ and let $Q_j=[\rho_{i+1}(x_j)]\in \Delta_{i+1}$ for $j=1,2$.  Then $P(Q_1)=P(Q_2)=[\rho_i(x_j)]$ for $j=1,2$, but $Q_1$ and $Q_2$ are disjoint.  Since $\overline{A}(x_j)\in K$, equation \eqref{eq:construct K} implies that $\overline{A}(x_1)\not \in \partial_{\tau 2^{-\frac{i+1}{2s}}}Q_1$, so   
  $$d(\overline{A}(x_1), G\setminus Q_1)\ge \tau 2^{-(i+1)(1+\frac{1}{2s})}.$$
  Thus
  $$d(\overline{A}(x_1),\overline{A}(x_2))\ge \tau 2^{-(i+1)(1+\frac{1}{2s})}.$$
  Since $d(x_1,x_2)=2^{-i+1}$,
  $$d(\overline{A}(x_1),\overline{A}(x_2))\ge \frac{1}{8} \tau d(x,y)^{1+\frac{1}{2s}}.$$
\end{proof}

Let $p\in M$.  Let $B$, $C$, and $L$ be as in Lemma~\ref{lem:transfer lemma}, and suppose that $L$ is small enough that $B_{L}(p)\subset B$.  Let $r>0$ be small enough that 
\begin{equation}\label{eq:choice r}
2r<L\qquad\text{ and }\qquad
 C r^{\frac{1}{s}} \le \frac{\tau}{80 C_2}
\end{equation}
and let $E\from \overline{\cT} \to G$ be the map 
\begin{equation}\label{eq:define E}
  E:=\delta_{\frac{r}{2C_2}}\circ \overline{A}.
\end{equation}
Then $\Lip(E)\le r (2C_2)^{-1}\Lip(A)\le r$.  

For every $x\in J$, there is a unique unit-speed geodesic $\lambda_x\from [0,1]\to \overline{\cT}$ connecting $v_0$ to $x$.  The composition $\alpha_x:=E\circ \lambda_x$ is a horizontal curve in $G$ of length at most $r$.  Let $k_0\in K$ be a density point of $K$ and let $j_0\in J$ be such that $\overline{A}(j_0)=k_0$.  Let $\gamma\from [0,1]\to M$ be the unique horizontal curve with the same controls as $\alpha_{j_0}$ such that $\gamma(1)=p$, and let $q_0:=\gamma(0)$.  Note that $d(p,q_0)\le r$.

For each $x\in J$, let $\hat{\alpha}_x\from[0,1]\to M$ be the unique horizontal curve with the same controls as $\alpha_x$ such that $\hat{\alpha}_x(0)=q_0$.  Then $\hat{\alpha}_{j_0}=\gamma$.  

Let $F\from \overline{\cT} \to M$ be the map such that $F(\lambda_x(t))=\hat{\alpha}_x(t)$ for all $x\in J$ and $t\in [0,1]$.  This is well defined, because if $x,y\in J$ and $\lambda_x(t)=\lambda_y(t)$, then $\lambda_x$ and $\lambda_y$ agree on the interval $[0,t]$, so $\hat{\alpha}_x$ and $\hat{\alpha}_y$ agree on the same interval.  By construction, $\hat{\alpha}_{j_0}=\gamma$, so
$$F(j_0)=F(\lambda_{j_0}(1))=\hat{\alpha}_{j_0}(1)=\gamma(1)=p.$$

We use $F$ to prove Proposition~\ref{prop:bilipschitz images}.

\begin{proof}[{Proof of Proposition~\ref{prop:bilipschitz images}}]
  Let $H\from K\to M$ be the map $H:=F\circ (\overline{A}|_J)^{-1}$.  We claim that $H$ is a biLipschitz embedding.

  Let $E:=\delta_{\frac{r}{2C_2}}\circ \overline{A}$ as in the definition of $A$.  Then $H=F\circ (E|_J)^{-1} \circ \delta_{\frac{2C_2}{r}}$, so $H$ is biLipschitz if and only if $F\circ (E|_J)^{-1}$ is biLipschitz.  Since $E$ is injective on $J$, it suffices to show that for all $x,y\in J$, we have
  \begin{equation}\label{eq:compare E F}
    \frac{1}{2}d_G(E(x),E(y)) \le d_M(F(x),F(y)) \le 2d_G(E(x),E(y)).
  \end{equation}
  
  Let $i=i(x,y)$ be as in \eqref{eq:define i}.  By \eqref{eq:i is log},  $d_\cT(x,y)=2^{-i+1}$.  By Lemma~\ref{lem:barA biholder}, 
  \begin{equation}\label{eq:x y separation}
    d_G(E(x),E(y))=\frac{r}{2C_2}d_G(\overline{A}(x),\overline{A}(y)) \ge \frac{\tau r}{16 C_2} d_\cT(x,y)^{1+\frac{1}{2s}}.
  \end{equation}

  Let $\alpha,\beta\from [0,1]\to \overline{\cT}$ be the curves $\alpha(t)=\lambda_x(1-2^{-i}+2^{-i}t)$, $\beta(t)=\lambda_y(1-2^{-i}+2^{-i}t)$.  These are geodesics such that $\alpha(0)=\beta(0)=a(x,y)$, $\alpha(1)=x$, and $\beta(1)=y$.  The compositions $E\circ \alpha$ and $E\circ \beta$ are horizontal curves in $G$ of length at most $r2^{-i}$ that both start at $E(a(x,y))$, and $F\circ \alpha$ and $F\circ \beta$ are curves in $M$ with the same controls that both start at $q:=F(a(x,y))$.  Since
  $$d(p,q)\le d(p,q_0)+d(q_0,q)\le r+(1-2^{-i})r \stackrel{\eqref{eq:choice r}}{<}L,$$
  we have $q\in B$, and we can apply Lemma~\ref{lem:transfer lemma}.

  Let $u_\alpha,u_\beta$ be the controls of $E\circ \alpha$ and $E\circ \beta$; we have $$\|u_\alpha(t)\|_1, \|u_\beta(t)\|_1 \le r2^{-i}.$$
  By Lemma~\ref{lem:transfer lemma}, 
  \begin{align*}
    |d_G(E(x),E(y))-d_M(F(x),F(y))| 
    &\le C (\|u_\alpha\|_1+\|u_\beta\|_1)^{1+\frac{1}{s}}\\ 
    &\le C (r d_\cT(x,y))^{1+\frac{1}{s}} \\ 
    &= C r^{\frac{1}{s}} \cdot r \cdot d_\cT(x,y)^{1+\frac{1}{s}} \\ 
    &\stackrel{\eqref{eq:choice r}}{\le} \frac{\tau}{80C_2} \cdot r \cdot 2 d_\cT(x,y)^{1+\frac{1}{2s}} \\ 
    &\stackrel{\eqref{eq:x y separation}}{\le} \frac{1}{2} d_G(E(x),E(y)).
  \end{align*}
  This implies \eqref{eq:compare E F}, so $H$ is biLipschitz.

  Finally, since $H$ is a biLipschitz map from $K\subset G$ to $M$, $k_0$ is a point of density of $K$, and $G$ and $M$ have the same Hausdorff dimension, the image $H(K)$ has positive density (with respect to top-dimensional Hausdorff measure) at $H(k_0)=F(\overline{A}^{-1}(\overline{A}(j_0)))=F(j_0)=p$.
\end{proof}

\section{Proof of Theorem~\ref{thm:Teorema_uno}}\label{sec:proof of main}

We break the proof of Theorem~\ref{thm:Teorema_uno} into two parts.

\subsection{$G$ is the tangent almost everywhere \(\Rightarrow\) $M$ is countably $G$-rectifiable}\label{sec:proof of main1}
First, we prove the reverse direction in Theorem~\ref{thm:Teorema_uno}. 
 It suffices to show that any compact subset of $M$ can be covered by the union of countably many bilipschitz images of subsets of $G$ and a null set.  Let $C\subset M$ be compact.

Let $L=\inf \{\mu(C\setminus \bigcup_{i=1}^\infty X_i)\}$, where the infimum is taken over countable sequences of sets $X_i\subset M$ such that $X_i$ is the bilipschitz image of a subset of $G$ (henceforth, a \emph{bilipschitz image}).  We claim that $L=0$.  First, note that this infimum is achieved by some sequence of sets $X_i$; if $X_{i,j}\subset M$ are bilipschitz images such that for any $j$,
$$\mu(C\setminus \bigcup_{i=1}^\infty X_{i,j})\le L+\frac{1}{j},$$
then $\mu(C\setminus \bigcup_{i=1}^\infty\bigcup_{j=1}^\infty X_{i,j})=L$.

By way of contradiction, suppose that $L>0$.  Let $X_1,X_2,\dots\subset M$ be bilipschitz images and let $S=C\setminus \bigcup_i X_i$ be such that $\mu(S)=L$.  Let $p\in S$ be a point such that $S$ has density $1$ at $p$.  By Proposition~\ref{prop:bilipschitz images}, there is a bilipschitz image $Y$ that has positive density at $p$.  Then
$$\mu(C\setminus (Y\cup \bigcup_i X_i))=\mu(S\setminus Y)<L,$$
which contradicts the minimality of $L$.  Therefore, $L=0$, and there is a sequence of bilipschitz images $X_1,X_2,\dots\subset M$ such that
$$\mu(C\setminus \bigcup_{i=1}^\infty X_{i})=0,$$
as desired.  Hence, one direction of Theorem~\ref{thm:Teorema_uno} is proved.

\subsection{$M$ is countably $G$-rectifiable \(\Rightarrow\) $G$ is the tangent almost everywhere}\label{sec:proof of main2}

The forward direction of Theorem~\ref{thm:Teorema_uno} follows from the work of Pansu. 

Assume $M$ is countably $G$-rectifiable and of Hausdorff dimension $Q$.  Since the measure 
$\mathcal H^Q $ is doubling then almost every point $p$ is a point of density 1 for the image $f(U)$ of a biLipschitz map $f \from U \to M$, with $U \subseteq X$ measurable. Since $G$ and $M$ are doubling we can pass to tangents at $f^{-1}(p)$ and $p$ and the metric tangent of $M$ at $p$ equals the  metric tangent of $F(U)$ at $p$, see \cite{LeDonne6}. 
Moreover,  by Gromov compactness (or using ultralimits) we get some induced   biLipschitz map from $G$ to $\Cone(M,p)$,  which may depends on rescalings of taking tangents. Since the space $\Cone(M,p)$ is a Carnot group isomorphic to $\Nil(M,p)$,   Pansu's version \cite{Pansu} of Rademacher differentiation   implies that $G$ and $\Nil(M,p)$ must be isomorphic as Lie groups. 
\qed

\section{Examples}\label{sec:Example}
\subsection{An example with no positive-measure set with constant tangent}\label{sec:Example1}

We give an example of an equiregular 
sub-Riemannian manifold on which the tangent is not constant on any set of positive measure. We thank Ben Warhurst for discussing this example with us.
\begin{prop}
There exists a 7--dimensional equiregular sub-Riemannian manifold $M$ foliated by 6--dimensional manifolds with the property that if $p$ and $q$ are two points in different leaves, then the tangents $\Nil(M,p)$ and $\Nil(M,q)$
are not isomorphic as Lie groups.
\end{prop}

\begin{proof}
Consider in  $\mathbb{R}^7$ the rank-3 distribution spanned by
\begin{align*}
X_1  &=\partial_1\\
X_2 &= \partial_2 + x_1 \partial_4 - x_1 x_3(-1+x_1)\partial_7\\
X_3 &= \partial_3 + x_2 \partial_5 - x_1 \partial_6 - x_1 x_2 x_1 \partial_7.\\
\end{align*}
This is a bracket-generating equiregular distribution. 

The nilpotentization at a point $(x_1,\dots,x_7)$ is given by a stratified Lie algebra with basis $(E_1,\dots,E_7)$ and non-trivial relations:
\begin{align}
[E_1, E_2]  &= E_4 &[E_2, E_3] &= E_5  &[E_1, E_3] &= -E_6 \nonumber \\
 [E_1, E_5] &= -E_7  &[E_2, E_6] &= 2 x_1 E_7  &[E_3, E_4] &= (1-2 x_1) E_7 \label{symbol0}\\
 [E_1, E_4] & = -2 x_3 E_7  &[E_1, E_6] &= 2 x_2 E_7 \nonumber
\end{align}

We consider the change of variables
\begin{eqnarray*}
e_1 = E_1 + \alpha E_2+ \beta E_3,    &
e_2 = E_2   ,&
e_3 = E_3   , \\
\end{eqnarray*}
with $\alpha= - \frac{2x_2}{1+2x_1}$ and $ \beta=   \frac{ x_3}{1-2x_1} $ so that 
$$2 x_2 + \alpha 2 x_1+\alpha = -2x_3 +\beta+\beta (1-2x_1)= 0.$$

This change of variables proves that the nilpotentization  \eqref{symbol0} is isomorphic to the Lie algebra 147E, denoted by $\mathfrak{g}^\xi$, with parameter
 $\xi=2 x_1$.
For $\xi\in (0 , 1/2)$, such algebras are pairwise non-isomorphic, see Remark~\ref{rem:g:e}. Hence, the distribution above, restricted to the  open set $\{0<x_1<1/4\}$, has constant tangents on the $x_1$-hyperplanes, and on two such hyperplanes the tangent is different.\end{proof}
 
 \begin{remark}\label{rem:g:e}
In general, we know from the classification of nilpotent Lie algebras of dimension 7, see \cite{Gong_Thesis}, that  the Lie algebras  $\mathfrak{g}^\xi$, $\xi\in\R$, from the above proof are such that
 $\mathfrak{g}^\xi$ and $\mathfrak{g}^\eta$ are isomorphic if and only if 
 $$ \eta \in \left\{ \, \xi,\quad  \frac{1}{\xi}, \quad  1-\xi, \quad  \frac{-1+\xi}{\xi}, \quad \frac{-1}{-1+\xi}, \quad  \frac{\xi}{-1+\xi} \, \right\}. $$
Hence, one actually has that  the distribution on $\R^7$ from the proof  has that the tangent is not constant on any set of positive measure. \end{remark}

\subsection{An example with  constant tangent almost everywhere, but not everywhere}\label{sec:Example2}
We give an example of an equiregular sub-Riemannian manifold on which the tangent is constant everywhere, but it is not constant everywhere. We initially observe that a similar example that is not equiregular is given by the Martinet distribution in $\R^3$:
$$ \partial_x,\qquad \partial_y+x^2 \partial_z.$$

As a variation of this example, we shall consider the rank-4 distribution on $\R^5$, with coordinates $x_1, y_1, x_2, y_2, z$, defined by the vector fields
$$X_1:= \partial_{x_1},\qquad Y_1 :=\partial_{y_1}+x_1^2 \partial_z,  \qquad  X_2:= \partial_{x_2},\qquad Y_2 :=\partial_{y_2}+{x_2} \partial_z.$$
The only non-trivial brackets 
 of these four vector fields are
$$[X_1, Y_1] = 2{x_1}  \partial_z,  \qquad  [X_2, Y_2]= \partial_z.$$
This vector fields span the tangent at every point, so the structure is equiregular.

On a full-measure open set the metric tangent is the same Lie group. In fact, if $x_1\neq 0$ then the nilpotentization is the second Heisenberg group, i.e., the basis $X_1, \frac{1}{x_1} Y_1,X_2, Y_2, \partial_z$ form a standard basis of such a group.

However, the nilpotentization is not the same at every point. Indeed, if $x_1=0$ then the nilpotentization is the direct product of $\R^2$ and the first Heisenberg group, i.e., the basis $X_1,   Y_1,X_2, Y_2, \partial_z$ form a standard basis of such a group.

\bibliography{general_bibliography_copy} 
 
\bibliographystyle{amsalpha}

\end{document}